\documentclass[11pt]{amsart}
\usepackage{cite}
\usepackage{pinlabel}
\usepackage{amsmath} 
\usepackage{cjhebrew}
\usepackage{graphicx} 
\setkeys{Gin}{keepaspectratio}
\usepackage{hyperref}
\usepackage{url}
\usepackage{bm}
\usepackage{mathabx}
\usepackage[left=1.25in,top=1in,right=1.25in,bottom=1in,head=.1in]{geometry}
\usepackage{xcolor}
\usepackage{tikz}
\usepackage{adjustbox}
\usepackage[normalem]{ulem} %
\usepackage{enumitem}

\usepackage{verbatim}
\newcommand{\items}{\begin{itemize}[leftmargin=25pt,rightmargin=15pt]
  \setlength\itemsep{2pt}}
\newcommand{\stopitems}{\end{itemize}}

\usepackage[textsize=scriptsize]{todonotes}

\usepackage{fancyhdr}
\newtheorem{theorem}{Theorem}[section] 
\newtheorem*{theorem*}{Theorem}
\newtheorem{lemma}[theorem]{Lemma}

\newtheorem{question}[theorem]{Question}
\newtheorem*{conjecture*}{Conjecture}
\newtheorem*{question*}{Question}
\newtheorem*{lemma*}{Lemma}
\newtheorem{proposition}[theorem]{Proposition}
\newtheorem{corollary}[theorem]{Corollary}
\newtheorem*{corollary*}{Corollary}
\theoremstyle{definition}
\newtheorem{definition}[theorem]{Definition}

\newtheorem{remark}[theorem]{Remark}

\newtheorem*{example*}{Example}
\newtheorem*{remark*}{Remark}
\newtheorem*{remarks*}{Remarks}
\newtheorem*{addenda*}{Addenda}
\newtheorem*{construction*}{Construction}

\newcommand{\RP}{\mathbb{RP}}
\newcommand{\CP}{\mathbb{CP}}

\newcommand{\sss}{S^2 \times S^2} 

\newcommand{\R}{\mathbb R}
\newcommand{\Z}{\mathbb Z}
\newcommand{\bc}{\mathbb C}

\renewcommand{\phi}{\varphi}

\newcommand{\cs}{\mathbin{\#}}
\newcommand{\cpone}{\bc P^1}

\newcommand{\cptwo}{\bc P^2}

\newcommand{\cptwobar}{\smash{\overline{\bc P}^2}}

\newcommand{\unred}[1]{ \ignorespaces}  

\title{Smoothly knotted surfaces that remain distinct after many internal stabilizations\\ \today}
\author[Dave Auckly]{Dave Auckly}
\address{Department of Mathematics\newline\indent Kansas State University\newline\indent  Manhattan,
Kansas 66506}
\email{dav@math.ksu.edu}
\thanks{The author was partially supported by Simons Foundation grant 585139, and NSF grant DMS-1952755. The result was obtained while the author was  at the Simons
Laufer Mathematical Science Institute partially funded by NSF grant DMS-1928930, during the Fall 2022 semester.  
Math.~Subj.~Class.~2010: 57M25 (primary), 57Q60 (secondary).}

\begin{document}
\setlength{\headheight}{12.0pt}.
\begin{abstract}Internal stabilization adds a trivial handle to an  embedded surface in a coordinate chart. It is known that any pair of smoothly knotted surfaces in a simply-connected $4$-manifold become smoothly isotopic after sufficiently many internal stabilizations. In this paper, we show that there is no upper bound on the number of internal stabilizations required. In fact, this behavior is fairly generic. The definition of a subtly smoothly knotted pair of surfaces is given and it is shown that many surfaces may be modified to obtain subtly knotted surfaces with large internal stabilization distance. Furthermore, it is shown that after stabilizing any $4$-manifold with contain topologicaly isotopic, smoothly related, non-isotopic copies of any $3$-manifold having positive first betti number. 
\end{abstract}
\maketitle
\section{Introduction}
We first describe the results presented in this paper and then discuss prior work in order to place the results in context. Our main result depends on the following standard relations between diffeomorphisms. 
\begin{definition}\label{d1}
A diffeomorphism $\alpha$ is said to be \emph{topologically pseudoisotopic to the identity} if there is  
a homeomorphism $G:I\times Z \to I\times Z$ so that $G(0,z) = (0,z)$, and $G(1,z) = (1,\alpha(z))$. It is said to be 
\emph{pseudoisotopic to the identity} if $G$ may be taken to be a diffeomorphism. It is said to be \emph{topologically isotopic to the identity} if the homeomorphism $G$ may taken to be level-preserving, i.e., 
$G(t,z) = (t,G_t(z))$. It is smoothly isotopic to the identity if, in addition, it may be taken to be a diffeomorphism. If $\alpha$ restricts to the identity on a distinguished open ball, the connected sum $\alpha\cs 1_{\sss}$ will be defined. In this case we say $\alpha$ is \emph{$1$-stably isotopic to the identity} if $\alpha\cs 1_{\sss}$ is isotopic to the identity. 
\end{definition}
It also depends on the following relation between surfaces, which defines a new notion of knotted surfaces.
\begin{definition}\label{subtle}
We call a pair of embeddings $j, j':\Sigma\hookrightarrow Z$ \emph{isotopic} if there is a diffeomorphism $\beta$ that is isotopic to the identity and satisfies $j' = \beta\circ j$. The pair is \emph{subtly knotted} if there is a diffeomorphism $\alpha$ that is pseudoisotopic, topologically isotopic, and $1$-stably isotopic to the identity so that $j'=\alpha\circ j$ and yet $j$ and $j'$ are not smoothly isotopic.
\end{definition}

\begin{remark}
Notice that subtle knotting is strictly a $4$-dimensional phenomena. If two  knots in $S^3$ are related by an orientation preserving homeomorphism, they are isotopic \cite{cerf}. Similarly two such knots with homeomorphic complements are isotopic \cite{GL}. In general $3$-manifolds there are many isotopically distinct knots related by orientation-preserving homeomorphisms, \cite{abdpr}.  
\end{remark}

By the tubular neighborhood theorem for any smooth embedding $j:\Sigma\to Z$ there is a real vector bundle (of rank two in our case) $\nu(\Sigma)$ and an embedding $\nu_j:\nu(\Sigma)\hookrightarrow Z$
so that $j = \nu_j\circ\sigma_0$ where $\sigma_0$ is the zero section. 
We define the  \emph{internal stabilization} of $j$ to be the restriction of $\nu_j$ to $\Sigma\cs T^2$ when we make the sum of pairs $(\nu(j),\Sigma)\cs(S^4,T^2)$ and identify $\nu(j)\cs S^4$ with $\nu(j)$. Here $T^2$ is an unknotted torus in $S^4$. This is denoted $j^T:\Sigma\cs T^2\hookrightarrow Z$. It is well defined up to smooth isotopy. The $g$-fold internal stablilization  will be denoted $j^{gT}:\Sigma\cs F_g\hookrightarrow Z$. Just as there are many different notions of stabilization, there are many different notions of distances between embedded surfaces, \cite{singh}. For this paper we use the following.
\begin{definition}
The \emph{internal stabilization distance} $\delta(j_0,j_1)$ between $j_0:\Sigma\to Z$ and $j_0:\Sigma\to Z$ is the minimal $g$ so that $j^{gT}_0$ and $j^{gT}_1$ are isotopic.
\end{definition}
\noindent 
We can now state our main theorem and corollaries.

\begin{theorem}\label{TA}
Given any simply-connected, stably smoothable, closed $4$-manifold $X$, there is a  number $N$ and a smooth manifold $Z$ 
homeomorphic to 
 $X\cs N(\sss)$ with the following property.  For any smooth, homologically essential embedding $j_0:\Sigma\hookrightarrow Z$ of a closed, oriented surface $\Sigma$, there is a collection of smooth embeddings $\{j_p:\Sigma\hookrightarrow Z\}$ so that $j_p$ and $j_q$ are subtly smoothly knotted. Furthermore, for any $g$ there is a $p$ depending on $g$, so that $j_0^{gT}$ and $j_p^{gT}$ are subtly smoothly knotted. In particular the internal stabilization diameter of any topological isotopy class of homologically essential surfaces in $Z$ is infinite. 
\end{theorem}

\begin{remark}
Notice that this theorem does not make any assumption about the fundamental group of the complement of the exterior of the initial surface, and it does not make any assumption about the self-intersection of the surface. If $X$ has a smooth structure to begin with we may take the smooth structure on $Z$ to be the one resulting from the smooth connected sum. In many cases, it suffices to stabilize the manifold $X$ just one time ($N=1$).
\end{remark}

There has been interest in surfaces bounding knots in punctured $4$-manifilds, \cite{MM, MMRS, MMP, pich}. In addition there has been interest in exotic co-dimension one embeddings, \cite{ARcd1,KMT}. Theorem~\ref{TA} and its proof imply corollaries relevant to these issues.

\begin{corollary}\label{anypi1}
There are subtly knotted surfaces in some $4$-manifold with any given (finitely-presented) fundamental group that remain subtly knotted after many internal stabilizations.
\end{corollary}

\begin{corollary}\label{anyknot}
Given any knot $K$ and any simply-connected, stably smoothable, closed $4$-manifold $X$, there is a  number $N$ and a smooth manifold $Z$ 
homeomorphic to 
 $X\cs N(\sss)$ with the following property. For any smooth, homologically essential (not zero in $H_2(Z\setminus B^4,S^3)$),  embedding $(j_0,K):(\Sigma,\partial\Sigma)\hookrightarrow (Z\setminus B^4,S^3)$ of a once punctured, oriented surface $\Sigma$, there is a collection of smooth embeddings $\{j_0:\Sigma\hookrightarrow Z\setminus B^4\}$ so that $j_p$ and $j_q$ are subtly smoothly knotted. Furthermore, for any $g$ there is a $p$ depending on $g$, so that $j_0^{gT}$ and $j_p^{gT}$ are subtly smoothly knotted. In particular the internal stabilization diameter of any topological isotopy class of homologically essential surfaces in $Z\setminus B^4$ is infinite. 
\end{corollary}

There are also versions for properly-embedded surfaces in many $4$-manifolds with arbitrary fundamental group, and non-empty boundary. The following is an application to smoothly knotted $3$-manifolds.

\begin{corollary}\label{cd1}
Given any $3$-manifold $Y$ with positive first betti number and and simply-connected, stably smoothable, closed $4$-manifold $X$, there is a  number $N$ and a smooth manifold $Z$ 
homeomorphic to 
 $X\cs N(\sss)$ so that there is an infinite family of smooth embeddings of $Y$ in $Z$ all of which are related by diffeomorphisms that are topologiclaly isotopic to the identity, but no two distinct ones are isotopic. 
\end{corollary}

\subsubsection{Prior Results}
Surgery theory reduced many problems about high-dimensional manifolds to problems about algebra. It does not apply to smooth $4$-manifolds because it is not always possible to find suitable Whitney disks to remove excess intersection. It is, however effective after stabilizing by taking the connected sum with a sufficient number of copies of
 $\sss$, \cite{CS,quinn:stable}. It is also effective in the topological category for $4$-manifolds with good fundamental groups. Indeed, Wall showed that any two homotopy equivalent, simply-connected smooth $4$-manifolds would become diffeomorphic after taking the connected sum with enough copies of $\sss$ \cite{wall:4-manifolds}. This process is called \emph{stabilization}. Later, Gompf and independently Kreck showed that any pair of smooth, orientable, homeomorphic $4$-manifolds would become diffeomorphic after enough stabilizations, \cite{gompf:stable}.
The existence of exotic smooth structures on $4$-manifolds was discovered between Wall's work and the work of Gompf and Kreck. Shortly after smoothly exotic surfaces, i.e., surfaces that are topologically isotopic, but smoothly distinct were constructed \cite{Finashin1,Finashin2,Finashin3,FKV,FintushelStern1,Kim,KimRuberman1,KimRuberman2, KimRuberman3,HS,Mark,AKMR}. There are many different notions of stabilization for knotted surfaces in addition to the one addressed in this paper. We informally summarize some of these notions using the notation of topological pairs. Given a $4$-manifold -- surface pair $(X,\Sigma)$, an external stabilization is $(X,\Sigma)\cs(\sss,\emptyset)$, an intermediate stabilization is $(X,\Sigma)\cs(\sss,S^2\times\text{pt})$, an internal stabilization is 
$(X,\Sigma)\cs(S^4,T^2)$, and a non-orientable internal stabilization is $(X,\Sigma)\cs(S^4,\R{P}^2)$,  Quinn and Perron  showed that smoothly knotted surfaces would become smoothly isotopic after sufficiently many external stabilizations, \cite{quinn:isotopy,perron:isotopy1,perron:isotopy2}; Gabai showed that that one intermediate stabilization suffices in great generality, \cite{Ga}. Baykur and Sunukjian \cite{baykur-sunukjian:stab} showed that such surfaces in simply-connected manifolds would become smoothly isotopic after enough internal stabilizations, and Kamada showed that they would become equivalent after sufficiently many non-orientable stabilizations \cite{Ka}. Surgery corresponding to ambient internal $1$-handle attachment is a non-local notion of stabilization that is also studied.

For a long time it was open whether one stabilization would always suffice in the various contexts. 
In 2021 Juh\'asz and Zemke proved that there are properly embedded, connected surfaces with non-empty boundary in the $4$-ball with  arbitrarily large ambient internal $1$-handle distance \cite{JZ}.  This paper obtains results about a more restrictive notion of smooth knotting and a more restrictive notion of internal stabilization for (possibly closed) surfaces in (possibly closed) $4$-manifolds.  In 2020 Lin showed that there were exotic diffeomorphisms that remained exotic after one stabilization \cite{Lin-need2}. Based on Lin's example, Lin and Mukherjee showed that there is a pair of disconnected collections of disks properly embedded into the $4$-ball 
that remain smoothly distinct after one external stabilization \cite{LinMuk}. In a remarkable paper this year, Kang showed that there is a pair of smooth structures on a compact $4$-manifold with non-empty boundary that  remain distinct after one stabilization \cite{Kang}. Based on this example, Hayden, Kang and Mukherjee proved that a closely related manifold contains a pair of embedded $2$-spheres that remain distinct after one external stabilization or one internal stabilization \cite{HKM}.

Theorem~\ref{TA} is a refinement and shift from prior investigations of knotted surfaces. The first examples of knotted surfaces could be distinguished by the smooth structure of the complement of the surface. Konno, Mukherjee, and Taniguchi are promoting the idea that the the two relative pairs associated to exotic embedding should be diffeomorphic as pairs, so the the exotic feature is the embedding, not the complement or the pair, \cite{KMT}. Definition~\ref{subtle} takes this idea a bit further.

Theorem~\ref{TA} is stronger than the work of \cite{HKM} in that it shows the the diameter is infinite whereas the former only provides an example where the internal stabilization distance is at least two. The arguments used in  \cite{HKM,JZ}
require the ambient $4$-manifolds to have non-empty boundary. The results in this paper will apply when the ambient $4$-manifold has empty or non-empty boundary. The  surfaces in \cite{JZ} are not related by a diffeomorphism (see Lemma~\ref{not-diff}, whereas the examples constructed in this paper are related by diffeomorphisms that are very close to the identity in three interesting ways. Since this theorem demonstrates that there are surfaces related by one external stabilization, that require many internal stabilizations to become isotopic, it is interesting to consider the converse.

\begin{question}
Are there surfaces related by one internal stabilization that require many external stabilizations to become isotopic?
\end{question}

\begin{remark} Notice that smooth isotopy implies the remaining conditions in Definition~\ref{d1}. It is interesting to consider the extent that the other conditions fail to  imply smooth isotopy. 
By Lemma~\ref{SItoPI} in section~\ref{sec2} $1$-stably isotopic implies pseudoisotopic, so the pseudoisotopy condition is redundant. We include it because we will give a separate proof of this property for our examples that will generalize to the case of higher-dimensional families. 

Similarly, for simply-connected $4$-manifolds a diffeomorphism pseudoisotopic to the identity is topologically isotopic to the identity, \cite{quinn:isotopy}. This means topologically isotopic is redundant in the simply-connected case. However, there are diffeomorphisms of $S^1\times D^3$ that are pseudoisotopic to the identity rel boundary, but not topologically isotopic to the identity.

Lin showed that the Dehn twist on the essential $3$-sphere in $K3\cs K3$ is pseudoisotopic, and topologically isotopic to the identity, but not $1$-stably isotopic to the identity, \cite{Lin-need2}. Watanabe constructed diffeomorphisms that are homotopic, but not isotopic to the identity \cite{watanabe:theta}. Recently, Budney and Gabai have constructed so-called barbell diffeomorphisms on $S^1\times D^3$ that are pseudoisotopic to the identity, but not topologically isotopic to the identity
\cite{budney-gabai-hyp}. This implies that the diffeomorphisms become isotopic after enough stabilizations. The diffeomorphisms are explicit, and Budney and Gabai are working on showing that some of them are $1$-stably isotopic to the the identity.
\end{remark}

\noindent{\bf Acknowledgements}
Conversations with Konno, Mukherjee, Ruberman, and Taniguchi at SL Math suggested the questions addressed in this paper. The extension of the results to cover surfaces of negative self-intersection uses an argument in unpublished work of the author with Kim, Melvin, and Ruberman. Conversations with Budney, Gabai, Piccirillo, Powell and Ruberman helped the author refine the results and presentation.

\section{Basic Construction}\label{sec2} We use the following conventions and definitions.
Homology will always be considered with integer coefficients. The algebraic intersection number defines a symmetric bilinear form on the second homology of a $4$-manifold. For closed $4$-manifolds this form is unimodular. The $4$-manifold is \emph{even} if this form only takes even values. It is \emph{odd} otherwise. The \emph{divisibility} of a non-zero homology class $\alpha$ is the largest integer $d$ for which there is a class $\beta$ with $\alpha = d\beta$. A class is \emph{primitive} if it has divisibility equal to one.  A class $\alpha$ is \emph{characteristic} if $\alpha\cdot\beta \equiv \beta\cdot\beta \ (\text{mod} \ 2)$ for all classes $\beta$. It is called \emph{ordinary} otherwise. By an embedded surface we will mean a smooth embedding $j:\Sigma\hookrightarrow Z$. Everything will be assumed to be smooth unless specifically stated otherwise. 

In this section we describe the fundamental construction, with enough detail for the following special case of the main theorem. In the following section we prove that the surfaces remain smoothly distinct even after many internal stabilizations. In later sections we will refine the construction to obtain our main results. Let $E(2)$ denote the K3 surface. This is the total space of an elliptic fibration $\pi:E(2)\to\cpone$. 
\begin{proposition}\label{A}
In $E(2)\cs(\sss)$ every ordinary, primative homology class of positive self-intersection is represented by an infinite collection of smoothly embedded spheres so that any distinct pair is subtly knotted and the diameter of the collection in the internal stabilization distance is infinite.  
\end{proposition}

Before presenting the proof of this proposition, we provide the proof that stable isotopy implies pseudoisotopy and the proof that knotted surfaces in $D^4$ cannot have diffeomorphic exteriors. 
\begin{lemma}\label{SItoPI}
If a diffeomorphism restricts to the identity on a ball and becomes isotopic to the identity after some number of stabilizations, then it is pseudoisotopic to the identity.
\end{lemma} 
\begin{proof}
Let $\alpha:Z\to Z$ be a diffeomorphism that is the identity on $D^4\subset Z$. It extends to $I\times Z$ trivially on the $I$-factor, and to the result of adding a collection of $5$-dimensional $2$-handles with attaching circles in $D^4$ as the identity on the handles. If the framings are chosen correctly the upper boundary will be $Z\cs n(\sss)$. By hypothesis $\alpha\cs 1_{n(\sss)}$ is isotopic to the identity, so there is a diffeomorphism $G: I\times (Z\cs n(\sss)) \to I\times (Z\cs n(\sss))$ that restricts to $\alpha\cs 1_{n(\sss)}$ on the lower boundary and the identity on the upper boundary. This extends via the identity to the manifold obtained by attaching $5$-dimensional $3$-handles attached to the upper boundary along the cocores of the $2$-handles that were added in the prior step. Stacking these diffeomorphisms gives a diffeomorphism on $I\times Z$ that restricts ot $\alpha$ on the lower boundary and the identity on the upper boundary. The large manifold is $I\times Z$ because the $2$ and $3$ handles cancel. The same proof works if $Z$ has boundary and $\alpha$ restricts to the identity on the boundary.
\end{proof}

The following lemma follows a comment from to the author from Lisa Piccirillo and was the subject of a Math Overflow post, \cite{mo}.
\begin{lemma}\label{not-diff}
If the  embeddings $j, j':\Sigma\to D^4$ are related by a diffeomorphism that is the identity on the boundary then the embeddings are isotopic rel boundary. 
\end{lemma} 
\begin{proof}
It suffices to prove the result for embeddings that are related by a diffeomorphism that is the identity in a neighborhood of the boundary. Furthermore, we may assume that $.1D^4\cap(j(\Sigma)\cup j'(\Sigma)) = \emptyset$.
Let $\alpha:D^4\to D^4$ be such a diffeomorphism with $j'=\alpha\circ j$. Let $H:I\times D^4 \to I\times D^4$ be an isotopy rel the collar of the boundary (which we take to be $D^4\setminus .9D^4$), that maps radial lines to themselves, has $H_0 = 1_{D^4}$ and $H_1(.1D^4) = .9D^4$. One sees that $H_s^{-1}\circ\alpha\circ H_s\circ j$ is an isotopy between $j'$ and $j$ which extends to an ambient isotopy by the isotopy extension theorem.
\end{proof}

A Gompf $N(2)$ nucleus is a regular neighborhood of a certain $2$-complex  that embeds into $E(2)$ \cite{gompf:nuc}. This $2$-complex is the union of a cusp fiber with a  section. In fact, $E(2)$ contains three disjoint such nuclei $N_1, N_2, N_3$. In each let $T_a$, $a=1, 2, 3$ denote a generic fiber in a nucleus. This is an embedded torus of self-intersection zero. Similarly, each nucleus contains a section $\sigma_a$, $a=1, 2, 3$. These are embedded spheres of self-intersection $-2$, \cite{gompf-mrowka}. Let $A_a$ denote the homology class represented by $T_a + \sigma_a$, and $B_a$ be the homology class represented by $T_a$. It is known that intersection form on the second homology of $E(2)$ is given by $-E_8\oplus -E_8\oplus H_1\oplus H_2\oplus H_3$ where $-E_8$ denotes the negative-definite, even, unimodular form of rank $8$, and $H_a$ is the hyperbolic form generated by $\{A_a,B_b\}$. Let $A_0$ denote the sphere $S^2\times\{*\}$ and $B_0$ denote the sphere $\{*\}\times S^2$ in $S^2\times S^2$. We will use the same notation to denote the homology classes. 

The first manifold and embedded surfaces we consider are 
\[
Z = E(2)\cs(\sss), \ \text{and} \ \varsigma_k:S^2\hookrightarrow Z,
\]
where $[\varsigma_k] = A_0 + kB_0$. The explicit embedding is obtained by resolving the intersections in the union of one copy of $A_0$ with $k$ parallel copies of $B_0$. Notice that a parallel copy of $B_0$ will intersect $\varsigma_k(S^2)$ in exactly one point, and $\varsigma_k(S^2)$ has self-intersection equal to $2k$. This demonstrates that the class $[\varsigma_k]$ is primitive, ordinary, and $\pi_1(Z\setminus\varsigma_k(S^2)) = 1$.

\begin{remark}\label{N2-comp}
For later extensions, it will be useful to notice that all of the arguments in this section apply equally well to $(E(2)\cs(\sss)) \setminus N_3$.
\end{remark}

By \cite{wall:diffeomorphisms} there is a diffeomorphism $\zeta:Z\to Z$ so that $\zeta_*A_0 = A_1$ and $\zeta_*B_0 = B_1$.
We define embeddings
\[
\iota_k:S^2\hookrightarrow Z, \quad \iota_k = \zeta\circ \varsigma_k\,.
\]

As subtly smoothly knotted surfaces are related by diffeomorphisms we need to construct some interesting diffeomorphisms. We begin with reflections in spheres of self-intersection $\pm 2$, which we now review. The antipodal map on the two-sphere induces an orientation-preserving map on the tangent bundle to the two-sphere. The unit tangent sphere bundle of the two-sphere may be identified with the three-dimensional rotation group. Indeed, given a point on the sphere together with a unit tangent vector, one may then append the cross-product of the radius with the tangent vector to arrive at an oriented, orthonormal $3$-frame. Of course, the sphere bundles of any radius are also identified with this rotation group. The derivative of the antipodal map acts as $(-I_{2\times 2})\oplus (1)$ on these spheres. The reflection in the self-intersection $+2$ sphere at the zero section of the unit disk bundle of the
tangent bundle to $S^2$, is a smooth map that acts as the antipodal map on the zero section, as  $(-I_{2\times 2})\oplus (1)$ on the sphere bundles of radii in $(0,1/2]$, it is then extended via a family of rotations connecting $(-I_{2\times 2})\oplus (1)$ to $I_{3\times 3}$ on spheres of radii in $[1/2,1]$. We denote the unit disk bundle to $S^2$ by $DT$ and this reflection by $R:DT\to DT$.

Given an embedded sphere $\varsigma:S^2\to X$ of self-intersection $+2$ (or $-2$) there is an orientation preserving (reversing) diffeomorphism of $DT$ to the normal bundle of $\varsigma$. 
The reflection in the sphere $\varsigma(S^2)$, (denoted $R_\varsigma$) is defined to be $R$ in the tubular neighborhood extended by the identity elsewhere. To construct interesting diffeomorphisms, we first construct interesting spheres. The interesting spheres in turn arise from exotic smooth structures. 

Let $E(2;2p+1)$ be the result of performing a multiplicity $(2p+1)$-log transform in the first nucleus $N_1$. Freedman's classification theorem implies that there is a homeomorphism $\psi_p:E(2;2p+1)\to E(2)$, \cite{freedman}. In fact we may take $\psi$ to be a homeomorphism relative to $N_3$, indeed even relative to the complement of $N_1$ and the transformed copy of $N_1$ which we will denote by $N(2;2p+1)_1$. 
\begin{remark}\label{log-trans-H}
It is worth taking a moment to identify specific generators in the homology. Figure~\ref{N2_2} is a figure of the transformed nucleus. Denote the class of the $0$-framed handle by $F_{2p+1}$, and the class of the $(-8p^2-6p-2)$-framed handle by $S$. The surface representing the class $F_{2p+1}$ is called the \emph{multiple fiber}.
One directly checks that the other intersection number is $F_{2p+1}\cdot S = 1$. This implies that the boundary is a homology sphere. Furthermore the boundaries of  $N(2;2p+1)_1$  and $N_1$ agree. The map taking $F_{2p+1}$ to $T_1$,  and $S$ to $\sigma_1 - (4p^2+3p)T_1$ induces an isomorphism between these even intersection forms and allows one to conclude via Freedman's work that there is homeomorphism rel boundary inducing this map on homology. 
\end{remark}

\begin{figure}[ht]
\labellist
\small\hair 2pt
\pinlabel {$-2$} [ ] at 23 87
\pinlabel{$2p\!+\!1$} [ ] at 82 27
\pinlabel {$-8p^2-6p-2$} [ ] at 164 27
\pinlabel{$0$} [ ] at 218 50
\endlabellist
\centering
\includegraphics[scale = 1]{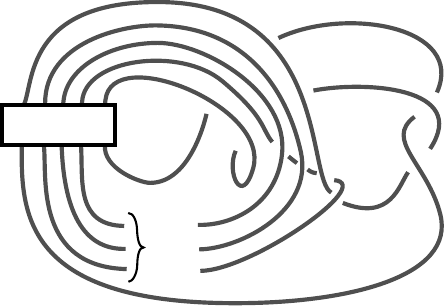}
  \caption{The Gompf nucleus $N_1$ after a $(p+1)$-log transform: (Figure by Ruberman)}\label{N2_2}.
\end{figure}
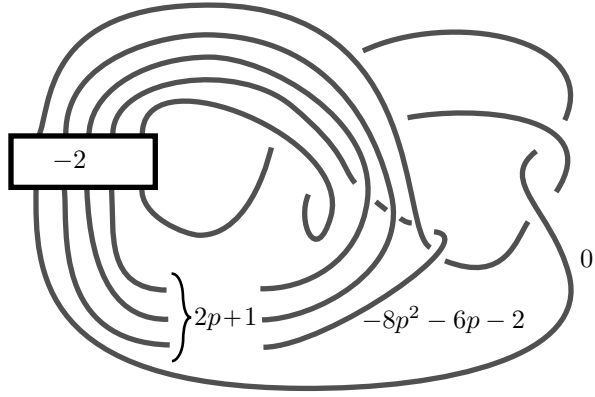

By a $5$-dimensional cobordism argument we note that there is a diffeomoprhism \hfill\newline $\varphi_p:E(2;2p+1)\cs(\sss)\to E(2)\cs(\sss)$, \cite{mandelbaum:irrational,mandelbaum,mandelbaum,gompf:elliptic}. By \cite{wall:diffeomorphisms} we may assume that 
$\psi_p\cs 1_\sss$ and $\varphi_p$ induce the same map on homology. We may even take $\varphi_p$ to be the identity outside of the sum of the nucleus with $\sss$.

The main diffeomorphisms we use to construct subtly knotted surfaces are the maps
\[
\alpha_p:Z\to Z, \quad \alpha_p = [\varphi_p, R_{\varsigma_1}R_{\varsigma_{-1}}],.
\]
We have 
\begin{align*}
\alpha_p &=  [\varphi_p, R_{\varsigma_1}R_{\varsigma_{-1}}] \\
&= \varphi_p R_{\varsigma_1}R_{\varsigma_{-1}} \varphi_p^{-1} R_{\varsigma_{-1}}^{-1}R_{\varsigma_1}^{-1} \\
&=  R_{\varphi_p\varsigma_1}R_{\varphi_p\varsigma_{-1}}  R_{\varsigma_{-1}}^{-1}R_{\varsigma_1}^{-1} \\
&=  R_{\varphi_p(\psi_p\cs 1_{\sss})^{-1}\varsigma_1}R_{\varphi_p(\psi_p\cs 1_{\sss})^{-1}\varsigma_{-1}}  R_{\varsigma_{-1}}^{-1}R_{\varsigma_1}^{-1}
\end{align*}
It follows from \cite{SKRAM} that $\varphi_p(\psi_p\cs 1_{\sss})^{-1}\varsigma_{1}$ and $\varsigma_1$ become smoothly isotopic after one external stabilization, i.e., taking the connected sum $Z\cs({\sss})$ away from the surfaces. The same holds for $\varphi_p(\psi_p\cs 1_{\sss})^{-1}\varsigma_{-1}$ and $\varsigma_{-1}$. It follows that $\alpha_p\cs 1_{\sss}$ is smoothly isotopic to the identity. 
\begin{remark}
The map $\varphi_p(\psi_p\cs 1_{\sss})^{-1}\varsigma_{1}$ is smooth even though $\psi$ is only a homeomorphism. the point is that $(\psi_p\cs 1_{\sss})^{-1}$ restricts to the identity on the image of $\varsigma_1$. We use analogues of this observation in other places in the construction.
\end{remark}
As remarked after the definition of subtly knotted, the stable isotopy implies that $\alpha$ is topologically isotopic and pseudoisotopic to the identity. One can see the same facts via the following argument.
Since $\varphi_p$ and $(\psi_p\cs 1_{\sss})^{-1}$ induce the same map on homology, work of Freedman and Quinn implies that $\varphi_p(\psi_p\cs 1_{\sss})^{-1}$ is topologically isotopic to the identity, \cite{freedman,quinn:isotopy}. It follows that $\alpha_p$ is topologically isotopic to the identity.  By \cite{sun} $\varphi_p(\psi_p\cs 1_{\sss})^{-1}\varsigma_{-1}$ is smoothly concordant to $\varsigma_{-1}$ since these are both genus zero surfaces in the same homology class. The same holds for $\varphi_p(\psi_p\cs 1_{\sss})^{-1}\varsigma_{1}$ and $\varsigma_1$. We conclude that $\alpha_p$ is pseudoisotopic to the identity. As we remarked earlier $\varsigma_1$ represents a primative, ordinary class, and has simply-connected complement. The same holds for 
$\varphi_p(\psi_p\cs 1_{\sss})^{-1}\varsigma_{1}$   which is in the same homology class. 

Notice that the fact that $E(2;1) = E(2)$ means we can take $\alpha_0$ to be the identity map. Now define
\[
\kappa_{k,p}:S^2\hookrightarrow Z, \ \text{by} \ \kappa_{k,p} = \alpha_p\circ \iota_k\,.
\]
Combine this with Wall's results that any pair of primitive, ordinary classes of the same self-intersection are related by an automorphism of the intersection form, \cite{wall:ortho}, and any automorphism of the intersection form of a once stabilized, indefinite manifold is represented by a diffeomorphism \cite{wall:diffeomorphisms}. Notice that when $j$ and $j'$ are subtly knotted with the required diffeomorphism denoted $\alpha$ and $f:Z\to Z$ is any diffeomrophism, then $f\circ j $ and $f\circ j'$ are smoothly knotted with the required diffeomorphim equal to $f\alpha f^{-1}$.  This completes the construction of the embeddings, the diffeomorphisms, topological isotopies, pseudoisotopies, and stabilized smooth isotopies required to establish Proposition~\ref{A}.

\section{Distinguishing the surfaces}
We use parameterized Seiberg-Witten gauge theory to distinguish the embeddings $\kappa_p$ even after many internal stabilizations as required to establish Proposition~\ref{A}. Let $\mathfrak{s}$ be a $\text{Spin}^c$ structure on $Z$ that is preserved by a diffeomorphism $\alpha$. Further assume that the expected dimension of the Seiberg-Witten moduli space is $-1$. In this case Ruberman defined
\[
SW_Z^{\pi_0,\Z_2}(\alpha,\mathfrak{s}) = \#\mathcal{M}^{SW}_{\mathfrak{s}}(\{g_t,\eta_t\}) \ \text{mod} \ 2,
\]
where $\{g_t,\eta_t\}$ is a generic family of metrics and perturbations satisfying $g_1 = \alpha^*g_0$ and $\eta_1 = \alpha^*\eta_0$, \cite{ruberman:swpos}. This invariant may be generalized in various ways. Call $c_1(\mathfrak{s})$ a $(\Z_2,\alpha)$-basic class if  
$SW_Z^{\pi_0,\Z_2}(\alpha,\mathfrak{s}) \neq 0$. One may also define the associated mapping torus by
\[
S^1\times_\alpha Z = I \times Z/\sim, \ \text{where} \ (1,z) \sim (0,\alpha(z)),.
\]
This is a family of $4$-manifolds. One may define a family Seiberg-Witten invariant for a bundle $X\hookrightarrow E \to B$ by
\[
FSW^{\Z_2}(E,\mathfrak{s}) = \#\mathcal{M}^{SW}_{\mathfrak{s}}(\{g_\theta,\eta_\theta\}) \ \text{mod} \ 2.
\]
A more precise description of this invariant and several refinements of it may be found in \cite{BK}. It should be clear from the above description that 
\[
SW_Z^{\pi_0,\Z_2}(\alpha,\mathfrak{s})  = FSW^{\Z_2}(S^1\times_\alpha Z,\mathfrak{s}). 
\]

Let  $\{g_t,\eta_t\}$ be a generic family of metrics and perturbations satisfying $g_1 = \alpha^*g_0$ and   $\eta_1 = \alpha^*\eta_0$. Let  $\{g'_t,\eta'_t\}$ be a generic family of metrics and perturbations satisfying $g'_1 = \beta^*g'_0$ and   $\eta'_1 = \beta^*\eta'_0$ with $g'_0 = \alpha^*g_0$ and $\eta'_0=\alpha^*\eta_0$. The first family is suitable for computing  
$SW_Z^{\pi_0,\Z_2}(\alpha,\mathfrak{s})$, the second is suitable for computing $SW_Z^{\pi_0,\Z_2}(\beta,\mathfrak{s})$. Taking the concatenation gives a family suitable for computing $SW_Z^{\pi_0,\Z_2}(\alpha\beta,\mathfrak{s})$ and allows one to conclude \cite{ruberman:swpos} that
\[
SW_Z^{\pi_0,\Z_2}(\alpha\beta,\mathfrak{s}) = SW_Z^{\pi_0,\Z_2}(\alpha,\mathfrak{s}) + SW_Z^{\pi_0,\Z_2}(\beta,\mathfrak{s})\,.
\]  
We conclude that 
\[
SW_Z^{\pi_0,\Z_2}(\alpha_p,\mathfrak{s}) = SW^{\pi_0,\Z_2}_Z( R_{\varphi_p\varsigma_1}R_{\varphi_p\varsigma_{-1}},\mathfrak{s}) + SW_Z^{\pi_0,\Z_2}( R_{\varsigma_1}R_{\varsigma_{-1}},\mathfrak{s}).
\]

It is clear from the description of $R_{\varsigma_{\pm 1}}$ that the induced map on homology is 
\[
R_{\varsigma_{\pm 1}*}(x) = x \mp (x\cdot [\varsigma_{\pm 1}])[\varsigma_{\pm 1}].
\]
It follows that $(R_{\varsigma_1}R_{\varsigma_{-1}})_*$ acts as $-1$ on $H^2(\sss)$ and hence acts as $-1$ on $H^2_+(\sss)$. Thus the first Stiefel-Whitney class of the associated bundle is 
\[
w_1(S^1\times_{(R_{\varsigma_1}R_{\varsigma_{-1}})_*}H^2_+(\sss))[S^1] = 1.
\] 
It follows from the main theorem of \cite{BK} that 
\[
SW_Z^{\pi_0,\Z_2}( R_{\varsigma_1}R_{\varsigma_{-1}},\mathfrak{s}) = FSW^{\Z_2}( S^1\times_{R_{\varsigma_1}R_{\varsigma_{-1}}}E(2)\cs (\sss),\mathfrak{s}) = SW^{\Z_2}(E(2),\mathfrak{s})\,, \ \text{and}
\]
\begin{small}
\[
SW_{E(2;2p+1)\cs\sss}^{\pi_0,\Z_2}( R_{\varsigma_1}R_{\varsigma_{-1}},\mathfrak{s}) = FSW^{\Z_2}( S^1\times_{R_{\varsigma_1}R_{\varsigma_{-1}}}E(2;2p+1)\cs (\sss),\mathfrak{s}) = SW^{\Z_2}(E(2;2p+1),\mathfrak{s})\,, 
\]
\end{small}
where $\mathfrak{s}$ restricts to the standard Spin structure on $\sss$.

Using the unique Spin structure on $E(2;2p+1)$, one may identify $\text{Spin}^c$ structures on $E(2;2p+1)$ with a subset of 
$H^2(E(2;2p+1))$. Let $2\mathfrak{t}$ represent both the cohomology class Poincar\'e dual to $2[T_1] = 2[F_{2p+1}]$ and the corresponding 
$\text{Spin}^c$ structure. (See Remark~\ref{log-trans-H}.) By \cite[Cor 1.4]{MMS}, $SW^{\Z_2}(E(2;2p+1),2\ell\mathfrak{t}) = 1$ precisely when $|\ell| \le p$. Recall that the cohomology classes with non-vanishing mod two Seiberg-Witten invariants are called \emph{mod two basic classes}. Thus, $2\ell\mathfrak{t}$,  $|\ell| \le p$ are the mod two basic classes for $E(2;2p+1)$. 

To put these results together, we use the following observation.
\begin{lemma}
If $\rho:X\to X$, and $\varphi:X\to Y$ are diffeomorphisms, there is an isomorphism of families
\[
\Phi: S^1\times_\rho X \to S^1\times_{\varphi\rho\varphi^{-1}} Y, \ \text{given by} \ \Phi([t,x]) = [t,\varphi(x)]. 
\] 
\end{lemma}
\noindent
This lemma gives
\begin{align*}
SW^{\pi_0,\Z_2}( &R_{\varphi_p\varsigma_1}R_{\varphi_p\varsigma_{-1}},2\ell\left((\psi_p\cs1_\sss)^{-1}\right)^*\mathfrak{t}) \\ 
& =
FSW^{\Z_2}(S^1\times_{R_{\varphi_p\varsigma_1}R_{\varphi_p\varsigma_{-1}}}E(2)\cs\sss,2\ell\left((\psi_p\cs1_\sss)^{-1}\right)^*\mathfrak{t}) \\
&= 
FSW^{\Z_2}(S^1\times_{R_{\varsigma_1}R_{\varsigma_{-1}}}E(2;2p+1)\cs\sss,2\ell\varphi_p^*\left((\psi_p\cs1_\sss)^{-1}\right)^*\mathfrak{t}) \\
&= 
SW^{\pi_0,\Z_2}({R_{\varsigma_1}R_{\varsigma_{-1}}},2\ell\mathfrak{t}) = SW^{\Z_2}(E(2;2p+1),2\ell\mathfrak{t}).
\end{align*}

It now follows from the homomorphism property that $SW^{\pi_0,\Z_2}(\alpha_p,2\ell\mathfrak{t}) = 1$ for $\ell$ with 
$0 < |\ell| \le p$. In particular, $c_1(2p\mathfrak{t})$ is a $(\Z_2,\alpha_p)$-basic class. Baraglia proved the following family adjunction inequality for one-dimensional families.
\begin{theorem}[Baraglia]\cite{Bar}\label{bar}
If $\alpha:Z\to Z$ is a diffeomorphism, and $\Sigma \subset Z$ is a surface of positive self-intersection that is smoothly isotopic to $\alpha(\Sigma)$, and $c$ is a $(\Z_2,\alpha)$-basic class, then
\[
2g(\Sigma) - 2 \ge \Sigma\cdot\Sigma + |c[\Sigma]|\,.
\]
\end{theorem}  

The prior section established that $\kappa_{k,0}$ and $\kappa_{k,p}$ are related by a diffeomorphism that is trivial in the three senses of topological isotopy, pseudoisotopy, and $1$-external stable smooth isotopy. It follows that the repeated internal stabilizations
$\kappa_{k,0}^{gT}$ and $\kappa_{k,p}^{gT}$ are related by the same diffeomorphism.

If two such stabilizations are isotopic, the family adjunction inequality gives
\[
2g - 2 \ge 2k + c_1(2p\mathfrak{t})[\iota_k] = 2k + 2p B_1\cdot (A_1 + kB_1) = 2(k+p)\,. 
\]
It follows that $g\ge k+ p +1$.  
We conclude that the embeddings $\kappa_{k,0}^{gT}$ and $\kappa_{k,p}^{gT}$  are subtly knotted for  $g\le k+ p$. This completes the proof of Proposition~\ref{A}.

\section{The main theorem}
In this section we collect and establish a few results that combine with the basic idea outlined in the prior two sections to give the proof of the main theorem. Since the homotopy type of a simply-connected $4$-manifold is determined by its intersection form we begin with a result
about indefinite unimodular quadratic forms. Let $H = \langle A, B \rangle$ denote the standard hyperbolic form with
$A^2 = B^2 = 0$ and $A\cdot B = 1$. In the diagonal form $9(+1)\oplus(-1)$ let $\{H_i\}_{i=1}^9 \cup \{E\}$ be the standard diagonal basis. The form $E_8$ is the positive-definite form obtained as the orthogonal complement to $3E+\sum_{i=1}^9H_i$. This is the only even unimodular form of rank and signature equal to eight. We use $nE_8$ to denote the direct sum of $n$ copies of $E_8$ when $n\ge 0$ and $|n|$ copies of the negative of $E_8$ when $n<0$. 
\begin{lemma}\label{std-form}
Let $Q$ be a unimodular, quadratic form with $r(Q)-|\sigma(Q)| \ge 6$. By the classification any such form contains a hyperbolic summand $\langle A_0, B_0 \rangle$. Given any element $v$ there is an automorphism $\xi_*\in \text{Aut}(Q)$ so that
$\xi_*v$ is orthogonal to  $\langle A_0, B_0 \rangle$. 
\end{lemma}
\begin{proof}
This follows from the classification of indefinite unimodular forms, \cite{MH}, combined with a theorem of Wall that establishes that
any two primitive, classes of the same square and type in a unimodular, quadratic form with $r(Q)-|\sigma(Q)| \ge 6$ are related by an automorphism, \cite{wall:ortho}. The result is clear for the zero class. Furthermore, it suffices to prove the result for primitive classes. Indeed if $v = dw$ and $\xi_*w$ is orthogonal to $\langle A_0, B_0 \rangle$, then $v$ is. Now consider the case when $Q$ is an even form.
By the classification this implies that $Q = nE_8\oplus mH$.  By assumption, we have at least two hyperbolic summands. Label the summands $\langle A_p, B_p \rangle$ for $p = 0, \dots, m-1$. Since any primitive element $v$ has a dual element $v^*$ so that  $v^*\cdot v = 1$, no primitive element may be characteristic. Any element has even square, and $A_1+kB_1$ are ordinary, primitive classes with square $2k$, so the desired automorphism exists by the result of Wall.  

Now let $Q$ be odd. 
By the classification of odd, indefinite, unimodular forms, we have $Q = p (1) \oplus q (-1)$. By the rank-signature inequality, $p, q \ge 3$ so one may write $Q = (p-2) (1) \oplus (q-2) (-1) \oplus 2H$. Let $\{H_i\}_{i=1}^{p-2} \cup \{E_i\}_{j=1}^{q-2}\cup \{A_0,B_0, A_1, B_1\}$ be a corresponding basis. Now any possible square is represented in the form
$(\epsilon E_1+A_1+kB_1)^2 = 2k - \epsilon$ for $\epsilon\in\{0,1\}$, and the classes $\epsilon E_1+A_1+kB_1$ are primitive and ordinary.

Finally, let $v$ be primitive and characteristic. This implies that $v^2 = \sigma + 8k$ for some integer $k$ \cite{MH}. 
Now $2A_1 + 2kB_1 + \sum_{i=1}^{p-2} H_i + \sum_{j=1}^{q-2} E_j$ is a characteristic element of square $\sigma +8k$, so the desired automorphism exists by the result of Wall.
\end{proof}

\begin{remark*}
Notice that the proof of this lemma gives us a standard form (up to automorphisms) for elements in unimodular quadratic forms with 
$r(Q)-|\sigma(Q)| \ge 6$. Namely, any non-zero class has a representative up to automorphism of one of the forms listed in Table~\ref{t:v-forms}. In this table we write an even unimodular form as $nE_8\oplus mH$ and let $\langle A_1, B_1 \rangle$ be a basis for one of the hyperbolic summands. We write an odd unimodular form as  $Q = (p-2) (1) \oplus (q-2) (-1) \oplus 2H$  with basis $\{H_i\}_{i=1}^{p-2} \cup \{E_i\}_{j=1}^{q-2}\cup \{A_0,B_0, A_1, B_1\}$.
\begin{table}[ht]
\renewcommand{\arraystretch}{1.5}
\centering
\begin{tabular}{llc}
{\bf Form}  & {\bf Element type} & {\bf Standard representative}\\
\hline
even  & ordinary  & $(2d+1)A_1 + (2d+1)kB_1$ \\
even  & characteristic & $2dA_1 + 2dkB_1$ \\
odd  & ordinary  & $d\epsilon E_1+dA_1 + dkB_1$ \\
odd  & characteristic  & $(2d+1)\left(2A_1 + 2kB_1 + \sum_{i=1}^{p-2} H_i + \sum_{j=1}^{q-2} E_j\right)$  \\
\end{tabular}
\\[2ex]
\caption{{Standard forms}}\label{t:v-forms}
\end{table}
\renewcommand{\arraystretch}{1}

\noindent
A second theorem of Wall \cite{wall:diffeomorphisms} asserts that the group of orientation preserving diffeomorphisms on an indefinite manifold that has been stabilized generates all automorphisms of the intersection form. 
\end{remark*}

Most smoothly knotted surfaces in smooth $4$-manifolds arise from exotic structures on closely related smooth $4$-manifolds. The study of the existence and non-uniqueness of smooth structures is a vast field with many results. There are many approaches to generate a collection of smooth $4$-manifolds suitable for the proof of our main theorem. We chose one approach that generates a collection of symplectic manifolds.
It is convenient to have a large collection of symplectic $4$-manifolds, each of which includes a self-intersection zero symplectic torus in an $N(2)$ nucleus. Such a collection was constructed in \cite{Park06}. We restate a special case of the main result of that paper here using wording optimized for for our application.

\begin{theorem}[Park]\label{geog} 
There is a constant $r_*$ so that any odd, unimodular form with rank $r(Q) > r_*$ and signature satisfying 
$|\sigma(Q)| \le \frac{3}{13}r(Q)$ and odd-dimensional maximal positive-definite subspace is represented by the intersection form of a simply-connected, closed, symplectic, 
$4$-manifold containing a self-intersection zero symplectic torus in an $N(2)$ nucleus. If, in addition, $\sigma(Q) \equiv 0 \ (\text{mod} \ 16)$, any even, unimodular  form with rank $r(Q) > r_*$ and signature satisfying 
$|\sigma(Q)| \le \frac{3}{13}r(Q)$ and odd-dimensional maximal positive-definite subspace is represented by  by the intersection form of a simply-connected, closed, symplectic, 
$4$-manifold containing  a self-intersection zero symplectic torus in an $N(2)$ nucleus.
\end{theorem}
\begin{proof}
Park's theorem does not consider the odd case, or establish that the relevant manifolds 
contain copies of $N(2)$. However we will see that these are both minor modifications to his result. His result covers the spin case and is stated using the square of the first Chern class $c$, and the holomorphic Euler characteristic $\chi$. These are related to the rank and signature by 
\[
c = 3\sigma + 2r +4, \ \text{and} \ \chi = \frac14(\sigma + r + 2).
\]
Park's theorem states that all but a finite number of allowable pairs $(c,\chi)$ satisfying $0 \le c \le m\chi$ are represented by an infinite number of smoothly distinct, simply-connected, spin, symplectic $4$-manifolds. Here $m$ is some constant larger than $8.76$. The upper bound on $\sigma$ in the theorem statement that we give is obtained by using $m = 8.75$. The lower bound given here is more restrictive than Park's result. We make this restriction in order to insure that the region is symmetric under
change of sign of the intersection form. The manifolds with odd intersection form may be obtained by repeated blow-up.
Park establishes the infinite collection of differential structures on these representatives via knot surgery on a self-intersection zero symplectic torus. 

The difference between $8.76$ and $8.75$ is large enough that we may first take a manifold with signature $16$ less than the desired manifold and then submanifold sum with a copy of $E(2)$ along a torus in one of the three nuclei without modifying the other nuclei.  This insures that the resulting manifold contains a Gompf-nucleus.
\end{proof}

A smooth $4$-manifold \emph{dissolves} if it has the form $p\cptwo\cs q\cptwobar$ or $nE(2) \cs m(\sss)$. Here we allow negative values of $n$ with the interpretation that a negative value refers to a sum of $|n|$ copies taken with the non-algebraic orientation. 
It follows directly from the classification of topological, simply-connected $4$-manifolds that every manifold in Theorem~\ref{geog} is homeomorphic to a manifold that dissolves. Wall's 
theorem
implies that  after taking the connected sum with enough copies of $\sss$ any smooth $4$-manifold will  dissolve \cite{wall:4-manifolds}.
Call a manifold \emph{almost completely decomposable} if the connected sum of the manifold with $\sss$ dissolves.

\begin{definition}
Call the collection of manifolds described in Theorem~\ref{geog} the \emph{Park manifolds}. 
\end{definition}

The use of the symplectic structure is that it provides an easy to state criteria for non-vanishing Seiberg-Witten invariants. Namely, according to Taubes, \cite{T-SW}, the Seiberg-Witten invariant of the canonical class on a symplectic manifold with $b^2_+>1$ satisfies $|\text{SW}(K)| = 1$. The symplectic torus of self-intersection zero in $N(2)$ is useful for submanifold sums, knot surgery, and  log transforms. The fact that the torus is in $N(2)$ allows one to easily conclude that log transform does not change the topological type of the manifold. 

\begin{lemma}\label{ZX}
Given any stably smoothable simply-connected $4$-manifold $X$, there is an integer $N$ and  a pair of symplectic, almost completely decomposable $4$-manifolds $Z_X'$ and $Z_{-X}'$ each containing a nucleus, so that $X\cs N(\sss)$ is homeomorphic to $Z_X = Z_X'\cs(\sss)$ which is orientation-reversing diffeomorphic to $Z_{-X} = Z_{-X}'\cs(\sss)$. If $X$ is smooth the homeomorphism may be taken to be a diffeomorphism.
\end{lemma}

\begin{proof}
Let $X$ be stably smoothable. By Freedman's homeomorphism classification of simply-connected $4$-manifolds,  there is an $n_0$ so that $X\cs n_0(\sss)$ is  homeomorphic to a Park manifold $Z''_X$ and $-X\cs n_0(\sss)$ is homeomorphic to a Park manifold $Z''_{-X}$. We need a pair of such manifolds that are almost decomposable. Since the family of Park manifolds is closed under submanifold sum along the distinguished tori, this can be arranged.  Indeed, by Theorem~\ref{geog} there is a simply-connected, spin, signature zero, symplectic $4$-manifold $P_0$, containing a self-intersection zero symplectic torus in an $N(2)$ nucleus. By \cite{wall:4-manifolds}, there is an $n_1>1$ so that $P_0\cs n_1(\sss)$, $Z_X''\cs n_1(\sss)$, and $Z_{-X}''\cs n_1(\sss)$ all dissolve. This implies that $P_0\cs n_1(\sss) \cong (n_1+b^2_+(P_0))(\sss)$. By \cite{mandelbaum:irrational,mandelbaum,mandelbaum,gompf:elliptic}, $N(2)\cs_T N(2) \cs (\sss) \cong N(2)\cs N(2) \cs 2(\sss)$. It follows that $Z_X' = Z_X'' \cs_T P_0^{n_1\cs_T}$ and $Z_{-X}' = Z_{-X}''\cs_T P_0^{n_1\cs_T}$ are almost completely decomposable,  Park manifolds.  The manifold from the theorem will be $Z_X = Z_X'\cs (\sss)$ and this is orientation reversing diffeomorphic to 
$Z_{-X} = Z_{-X}'\cs (\sss)$.
The number $N$ from the theorem statement will just be $n_1b^2_+(P_0) + n_1 + n_0+1$ because
\begin{align*}
Z_X &= Z_X'\cs (\sss) \\
&= Z_X'' \cs_T P_0^{n_1\cs_T}\cs (\sss) \\
&= Z_X'' \cs P_0^{n_1\cs}\cs (n_1+1)(\sss) \\
&= Z_X'' \cs  (n_1b^2_+(P_0) +n_1+1)(\sss)  = X\cs N(\sss).
\end{align*} 
The analogous computation shows $Z_{-X} = -X\cs N(\sss)$. If $X$ is already smooth we can pick $n_0$ so that $X\cs n_0(\sss)$ is diffeomorphic to $Z''_X\cs(\sss)$ for some Park manifold $Z''_X$, and $-X\cs n_0(\sss)$ is diffeomorphic to $Z''_{-X}\cs(\sss)$ for some Park manifold $Z''_{-X}$.
\end{proof}

\begin{proof}[Proof of Theorem~\ref{TA}]
Let $Z_{\pm X}$ be the manifolds given by Lemma~\ref{ZX}, and
let $j_0:\Sigma\to Z_X$ be any homologically essential embedding with non-negative self-intersection. The case of negative self-intersection will be addressed at the end of the proof. From here the proof will divide into cases depending if $j_0(\Sigma)$ is characteristic or ordinary and the type of the intersection form of $X$. First consider the case when $X$ is Spin and $j_0(\Sigma)$ is characteristic.

As in the proof of the special case of $E(2)\cs(\sss)$ let $\sigma$ and $T$ be the section and fiber in the nucleus in $Z_X'$ and let $A_1$ denote the homology class of $\sigma+T$ and $B_1$ denote the homology class of $T$ in $Z_X'$. Let $A_0$ denote the homology class of $S^2\times \text{pt}$ and $B_0$ denote the homology class of $\text{pt}\times S^2$ in the $\sss$ summand of $Z_X = Z_X'\cs (\sss)$. From the proof of Lemma~\ref{std-form} and Table~\ref{t:v-forms} there is an automorphism $\zeta_*$ of the intersection form so that  $\zeta_*[j_0(\sigma)] = 2dA_1 + 2dkB_1$ with $k\ge 0$ and $d>0$.  Furthermore, by a theorem of Wall, there is a diffeomorphism $\zeta$ realizing $\zeta_*$ at the level of intersection forms. As diffeomorphisms take  pairs of subtly knotted surfaces to pairs of subtly knotted surfaces, we can assume without loss of generality that $\zeta = 1$.

Let $Z_{X;2p+1}'$ denote the result of performing a $(2p+1)$-log transform in the distinguished nucleus of $Z_X'$. The homeomorphism relating the nucleus and the log-transformed nucleus extends via the identity to a homeomorphism $\psi_p: Z_{X;2p+1}' \to Z_{X}'$. Similarly the stable diffeomorphism relating the sum of the log-transformed nucleus and $\sss$ with the sum of the nucleus and $\sss$ extends via the identity to a diffeomorphsim
\[
\varphi_p: Z_{X;2p+1}' \cs(\sss) \to Z_{X}'\cs (\sss).
\]
As in the proof of Proposition~\ref{A} define
\[
\alpha_p:Z_X\to Z_X, \quad \alpha_p = [\varphi_p, R_{\varsigma_1}R_{\varsigma_{-1}}],.
\]
Here $\varsigma_\pm$ are the maps with image in the $\sss$ factor of $Z_X = Z_X'\cs(\sss)$. The same argument shows that
$\alpha_p\cs 1_{\sss}$ is isotopic to the identity, hence $\alpha_p$ is pseudoisotopic and topologically isotopic to the identity. Set $j_p = \alpha_p\circ j_0$.  

Now turn to distinguishing $j_0^{gT}$ and $j_p^{gT}$. Since $Z_X'$ is symplectic, it has a canonical class $K$ with $SW^{\Z_2}(Z_X',K) = 1$. Since the classes $A_1$ and $B_1$ are each represented by a torus in $Z_X'$, the adjunction inequality gives
\[
|K\cdot A_1| = A_1\cdot A_1 + |K\cdot A_1| \le 0, \ \text{and} \ |K\cdot B_1| = B_1\cdot B_1 + |K\cdot B_1| \le 0.
\]
The adjunction inequality also implies that $SW^{\Z_2}(Z_X',K+2\ell\mathfrak{t}) = 0$ for $\ell \neq 0$.
By \cite[Cor 1.4]{MMS}, $SW^{\Z_2}(Z_{X;2p+1}' ,K+2\ell\mathfrak{t}) = 1$ precisely when $|\ell| \le p$ where  $2\mathfrak{t}$ represents both the cohomology class Poincar\'e dual to $2[T_1] = 2[F_{2p+1}]$ and the corresponding 
$\text{Spin}^c$ structure.

Now,
\[
SW_{Z_X'}^{\pi_0,\Z_2}(\alpha_p,K+2\ell\mathfrak{t}) = SW^{\pi_0,\Z_2}_{Z_X'}( R_{\varphi_p\varsigma_1}R_{\varphi_p\varsigma_{-1}},K+2\ell\mathfrak{t}) + SW_{Z_X'}^{\pi_0,\Z_2}( R_{\varsigma_1}R_{\varsigma_{-1}},K+2\ell\mathfrak{t}),
\]
and
\[
SW^{\pi_0,\Z_2}_{Z_X'}( R_{\varphi_p\varsigma_1}R_{\varphi_p\varsigma_{-1}},K+2\ell\mathfrak{t}) = 
SW^{\pi_0,\Z_2}_{Z_{X,2p+1}'}( R_{\varsigma_1}R_{\varsigma_{-1}},K+2\ell\mathfrak{t}).
\]

Using the main theorem of \cite{BK} gives 
\[
SW_Z^{\pi_0,\Z_2}( R_{\varsigma_1}R_{\varsigma_{-1}},\mathfrak{s}) = FSW^{\Z_2}( S^1\times_{R_{\varsigma_1}R_{\varsigma_{-1}}}Z_X,\mathfrak{s}) = SW^{\Z_2}(Z_X',\mathfrak{s})\,, \ \text{and}
\]
\[
SW_{E(2;2p+1)\cs\sss}^{\pi_0,\Z_2}( R_{\varsigma_1}R_{\varsigma_{-1}},\mathfrak{s}) = FSW^{\Z_2}( S^1\times_{R_{\varsigma_1}R_{\varsigma_{-1}}}Z_{X,2p+1},\mathfrak{s}) = SW^{\Z_2}(Z_{X,2p+1}',\mathfrak{s})\,, 
\]
where $\mathfrak{s}$ restricts to the standard Spin structure on $\sss$.
Thus,
$SW^{\pi_0,\Z_2}(\alpha_p,2\ell\mathfrak{t}) = 1$ for $\ell$ with 
$0 < |\ell| \le p$. In particular, $c_1(2p\mathfrak{t})$ is a $(\Z_2,\alpha_p)$-basic class. By the $1$-parameter adjunction inequality of Theorem~\ref{bar}, the assumption that $j_0^{gT}$ and $j_p^{gT}$ are isotopic gives
\begin{align*}
2(g+g(\Sigma)) - 2 &\ge [j_0(\Sigma)]\cdot [j_0(\Sigma)] +  |[j_0(\Sigma)]\cdot [K+2p B_1]| \\
&\ge  (2dA_1 + 2dkB_1)\cdot (2dA_1 + 2dkB_1) + |(2dA_1 + 2dkB_1)\cdot (K+2p B_1)| \\
 &= 8d^2k + 4dp. 
\end{align*}
It follows that $j_0^{gT}$ and $j_p^{gT}$ are subtly knotted when $4dp > 2(g+g(\Sigma)) - 2 - 8d^2k$.

If $X$ is Spin, $j_0(\Sigma)$ has non-negative self-intersection, and is ordinary, 
 the homology class may be taken to be $[j_0(\Sigma)] =  (2d+1)A_1 + (2d+1)kB_1$ with $d, k, \ge 0$. This gives
\begin{align*}
2(g+g(\Sigma)) - 2 &\ge  ((2d+1)A_1 + (2d+1)kB_1)\cdot ((2d+1)A_1 + (2d+1)kB_1) \\
&\qquad + |((2d+1)A_1 + (2d+1)kB_1)\cdot (K+2p B_1)| \\
 &= 8d^2k + 8dk + 2k + (4d+2)p, 
\end{align*}
so that $j_0^{gT}$ and $j_p^{gT}$ are subtly knotted when $(4d+2)p > 2(g+g(\Sigma)) - 2 - 8d^2k - 8dk - 2k$.

If $X$ is not Spin, $j_0(\Sigma)$ has non-negative self-intersection, and is ordinary, 
 the homology class may be taken to be $[j_0(\Sigma)] =  dA_1 + dkB_1+d\epsilon E_1$ with $d > 0$ and $0\le\epsilon \le k$. This gives
\begin{align*}
2(g+g(\Sigma)) - 2 &\ge  (dA_1 + dkB_1+d\epsilon E_1)\cdot (dA_1 + dkB_1+d\epsilon E_1) \\
&\qquad + |(dA_1 + dkB_1+d\epsilon E_1)\cdot (K+2p B_1)| \\
 &\ge 2d^2k - d^2\epsilon +|d\epsilon K\cdot E_1 + 2dp| 
\ge 2d^2k - d^2\epsilon - d\epsilon |K\cdot E_1| + 2dp, 
\end{align*}
so that $j_0^{gT}$ and $j_p^{gT}$ are subtly knotted when $2dp > 2(g+g(\Sigma)) - 2 - 2d^2k + d^2\epsilon + d\epsilon |K\cdot E_1|$.

If $X$ is not Spin, $j_0(\Sigma)$ has non-negative self-intersection, and is characteristic, 
 the homology class may be taken to be $[j_0(\Sigma)] =  (2d+1)\left(2A_1 + 2kB_1 + \sum_{i=1}^{r-2} H_i + \sum_{j=1}^{s-2} E_j\right)$ with $d \ge 0$ and $k \ge (s-r)/8$. This gives
\begin{align*}
2(g+g(\Sigma)) - 2 &\ge  \left((2d+1)\left(2A_1 + 2kB_1 + \sum_{i=1}^{r-2} H_i + \sum_{j=1}^{s-2} E_j\right)\right)^2
\\
&\qquad + |(2d+1)\left(2A_1 + 2kB_1 + \sum_{i=1}^{r-2} H_i + \sum_{j=1}^{s-2} E_j\right)\cdot (K+2p B_1)| \\
 &\ge (2d+1)^2(8k +r -s) +|(2d+1)\left(\sum_{i=1}^{r-2} H_i + \sum_{j=1}^{s-2} E_j\right) \cdot K  + (8d+4)p| \\
 &\ge (2d+1)^2(8k +r -s) - (2d+1)|\left(\sum_{i=1}^{r-2} H_i + \sum_{j=1}^{s-2} E_j\right) \cdot K|  + (8d+4)p, 
\end{align*}
so that $j_0^{gT}$ and $j_p^{gT}$ are subtly knotted when 
\[
(8d+4)p > 2(g+g(\Sigma)) - (2d+1)^2(8k +r -s) + (2d+1)|\left(\sum_{i=1}^{r-2} H_i + \sum_{j=1}^{s-2} E_j\right) \cdot K|.
\]

Notice that this also shows that there is an infinite collection of subtly knotted surfaces. Since any pair of topologically isotopic surfaces become isotopic after some number of internal stabilizations there is a sequence $p_k$ so that the internal stabilization distance $\delta(j_0,j_{p_k})$ is strictly increasing. Since this distance only depends upon the isotopy type of $j_{p_k}$ we conclude that these are all smoothly distinct. 

The only remaining item to check is the case when $j_0(\Sigma)$ has negative square. Here one can take an orientation reversing diffeomorphism $f:Z_X \to Z_{-X}$ and notice that $\kappa_0(\Sigma) = f\circ j_0(\Sigma)$ has non-negative square. Thus there is an infinite family  of subtly knotted maps $\kappa_p$ so that $\delta(\kappa_0,\kappa_p)$ tends to infinity. Since the properties of subtle knotting and internal stabilization distance are preserved by diffeomorphisms, the family $j_p = f^{-1}\circ\kappa_p$ satisfies the conclusion of the theorem. 
\end{proof}

\section{Extensions}
It would not be unreasonable to conjecture that the main theorem holds without restriction on the (finitely-presented)  fundamental group for any $4$-manifold. To follow the proof given in the simply-connected case, one would need to know the analogue of Lemma~\ref{ZX}. Knowing that any stably smoothable $4$-manifold was stably symplectic would suffice. In addition, one would need to classify elements of the second homology of a suitably stabilized copy of the manifold up to orientation-preserving diffeomorphism.

\begin{proof}[Proof of Corollary~\ref{anypi1}]
According to \cite[Theorem 4.1]{gompf:symp}, any finitely-presented group is the fundamental group of a closed, symplectic $4$-manifold, $X$. Looking at the proof one sees that these manifolds are obtained via submanifold sum starting with a product $\Sigma_g\times T^2$ and many copies of $E(2)$. In particular, these manifolds each contain a nucleus. As we have been doing let $B_1$ denote the class of the torus, and $A_1$ denote the class of $T+\sigma$. According to \cite{T-SW} the canonical class of this manifold will be a $\Z_2$ basic class, and according to the adjunction inequality we will have $A_1\cdot K = B_1\cdot K =0$. Let $X_{2p+1}$ be the result of a $(2p+1)$-log transform on $X$. The formula of \cite{MMS} still implies that $K+2\ell\mathfrak{t}$ are $\Z_2$ basic classes of $X_{2p+1}$. Let $Z = X\cs (\sss)$ and let $j_0:\Sigma \hookrightarrow$ be any embedding with $[j_0(\sigma)]\cdot B_1 \neq 0$ with image in $(E(2)\setminus N(2))\cs(\sss)$. Proceeding in as in the proof of the main theorem one constructs the homeomorphsim $\psi_p$ and stable diffeomorphism $\alpha_p$. Notice that the diffeomorphism $\alpha_p$ has support in $(E(2)\setminus N(2))\cs(\sss)$, so the main theorem of \cite{SKRAM} still applies and one sees that $\alpha_p$ becomes trivial after one external stabilization. The computation of the family invariant and application of the $1$-parameter adjunction inequality apply as in the proof of the main theorem.  There is one more difference between the simply-connected case and the non-simply-connected case. The result of Baykur and Sunukjian establishes that topologically isotopic embeddings become smoothly isotopic after sufficiently many $1$-handle stabilizations. When the fundamental group of the boundary of a tubular neighborhood surjects to the fundamental group of the exterior of the surface, this implies that the surfaces also become smoothly isotopic after sufficiently many internal stabilizations. This is the case for surfaces in simply-connected manifolds. It will also be the case here as we arranged the embeddings, and diffeomorphisms to all take place in a simply-connected portion of the ambient manifold.
\end{proof}

\begin{proof}[Proof of  Corollary~\ref{anyknot}] The diffeomorphisms $\alpha_p$ used to construct the embeddings in the proof of Theorem~\ref{TA} all restrict to the identity on a distinguished disk. Let $Z\setminus B^4$ be the complement of  the interior of this disk, and let $(j_0,K):(\Sigma,\partial\Sigma)\hookrightarrow (Z\setminus B^4,S^3)$ be any proper embedding as in the statement of the corollary. Capping off the surface with a Seifert surface and $Z$ with the disk produces an embedding of a closed surface  in $Z$ as in the main theorem. Since the diffeomorphisms $\alpha_p$ are all trivial in a neighborhood of $B^4$, this ball and the Seifert surface may be removed to obtain the required surfaces with boundary. If two of the surfaces with boundary were isotopic, it would imply that the closed surfaces were isotopic, so that they would correspond to the same $p$-value.  
\end{proof}

The results from Corollaries~\ref{anyknot} and~\ref{anypi1} may be combined. Given any symplectic manifold containing a nucleus one may take log-transforms and create interesting diffeomorphisms in the once stabilized manifold. These diffeomorphims will have support in $N(2)\cs(\sss)$. Thus one can look for any embedded surface with non-trivial intersection number with the torus in the nucleus. One can then find a separating $3$-manifold $Y$ in the complement of $N(2)\cs(\sss)$ and excise the component of the exterior of $Y$ that is disjoint from $N(2)\cs(\sss)$ from the ambient manifold and excise the 
corresponding part of the surface. The result will be a properly embedded surface in a $4$-manifold with boundary and interesting fundamental group. The basic argument still applies in this situation. While this is fairly general, it is not completely general. Indeed, it is known that there are $3$-manifolds that do not embed into any symplectic $4$-manifold, \cite{DLM}.

\begin{proof}[Proof of Corollary~\ref{cd1}] It is known that any $3$-manifold $Y$ is the boundary of a Spin $2$-handlebody $W$, \cite{kaplan}. The homology sequence of the pair $(W,Y)$ implies that $H_2(Y)$ injects into $H_2(W)$. The Mayer–Vietoris sequence implies that $H_2(W)$ injects into $H_2(DW)$ where $DW = W\cup_Y -W$ is the double.
Furthermore, since $W$ is a Spin $2$-handlebody, the double is diffeomorphic to the connected sum of several copies of $\sss$. By taking the submanifold sum of $Z'_X$ of $P_0$, we can assure that $Z$ has enough $\sss$-summands that $j:Y\hookrightarrow Z$ embeds. Since $b_1(Y)>0$ there is a homologically essential surface $i:\Sigma\hookrightarrow Y$. This becomes a homologically essential surface in $Z$, so the main theorem may be applied to $j\circ i$ to obtain a family of diffeomorphisms $\alpha_p$ topologically isotopic to the identity so that $\alpha_p\circ j\circ i$ are distinct up to smooth isotopy.
This implies that $\alpha_p\circ j$ are distinct up to smooth isotopy.
\end{proof}

\def\cprime{$'$}
\providecommand{\bysame}{\leavevmode\hbox to3em{\hrulefill}\thinspace}

\end{document}